\documentclass[%
amsmath,amssymb,
preprint,%
]{revtex4-1}

\usepackage{graphicx}
\usepackage{dcolumn}
\usepackage{bm}

\usepackage[utf8]{inputenc}
\usepackage[T1]{fontenc}
\usepackage{mathptmx}
\usepackage{etoolbox}
\usepackage{amssymb,amsmath,amsthm}
\usepackage{mathrsfs}
\usepackage{hyperref}
\usepackage{bm}
\usepackage{color}
\usepackage{xcolor}
\usepackage{epstopdf}

\newtheorem{theorem}{Theorem}[section]

\newtheorem{lemma}[theorem]{Lemma}

\newtheorem{remark}[theorem]{Remark}

\usepackage{subfigure}

\makeatletter
\def\@email#1#2{%
 \endgroup
 \patchcmd{\titleblock@produce}
  {\frontmatter@RRAPformat}
  {\frontmatter@RRAPformat{\produce@RRAP{*#1\href{mailto:#2}{#2}}}\frontmatter@RRAPformat}
  {}{}
}%
\makeatother
\begin{document}

\preprint{ }

\title[Dynamics of a diffusive predator-prey system]{Dynamics of a diffusive predator-prey system with fear effect in advective environments}

\author{Daifeng Duan}

\affiliation{School of Science, Nanjing University of Posts and Telecommunications, Nanjing 210023, P.R.China.}

\author{Ben Niu*}
 \email{niu@hit.edu.cn}
\affiliation{Department of Mathematics, Harbin Institute of Technology, Weihai, Shandong 264209, P.R.China.}

\author{Yuan Yuan}%

\affiliation{Department of Mathematics and Statistics, Memorial University of Newfoundland,
St. John's, NL A1C 5S7, Canada. 
}%

\date{\today}

\begin{abstract}
We explore a diffusive predator-prey system that incorporates the fear effect in advective environments. Firstly, we analyze the eigenvalue problem and the adjoint operator, considering Constant-Flux and Dirichlet (CF/D) boundary conditions, as well as Free-Flow (FF) boundary conditions. Our investigation focuses on determining the direction and stability of spatial Hopf bifurcation, with the generation delay $\tau$ serving as the bifurcation parameter. Additionally, we examine the influence of both linear and Holling-II functional responses on the dynamics of the model. Through these analyses, we aim to gain a better understanding of the intricate relationship between advection, predation, and prey response in this system.
\end{abstract}

\maketitle

\section{\label{sec1}Introduction}
The predator-prey system, which involves the interaction of different populations within a single area, is widely recognized as a crucial component of the biological system. This system is instrumental in explaining various aspects of growth interaction, competition, and coevolution between species. Furthermore, all species rely on acquiring adequate resources from their environment to ensure their survival and avert extinction.
Recent studies have shed light on various types of interspecific interactions within predator-prey systems, with particular emphasis on the "fear effect" \cite{Lima1998Nonlethal,Creel2008Relationships,Lima2009Predators}. Fear, which is experienced by both humans and organisms, involves psychological activity and emotions. It is characterized by intense biological physiological tissue contractions, a notable increase in tissue density, rapid energy release, and an emotional response aimed at evading potential threats. While predators directly kill their prey, there is a school of thought suggesting that the presence of predators can significantly alter the behavior of prey, primarily driven by fear of predation risks. This indirect impact has been largely overlooked, with most functional response models focusing solely on direct predation. However, Wang et al. \cite{Wang2016Modelling}  proposed a pioneering mathematical model that considers the decrease in the birth rate of prey populations due to predation risk concerns. The model is as follows:
\begin{eqnarray}\label{523a}
\begin{cases}
    \dot{u}=u\big[r_{0}f(K,v)-d-au\big]-g(u)v, \\
    \dot{v}=v\big(-r_{2}+c g(u)\big).
\end{cases}
\end{eqnarray}
Following Wang et al.'s initial work, subsequent studies have explored various predator-prey models that incorporate fear costs into prey reproduction \cite{Duan2019Hopf,Wang2022Spatiotemporal}.
These studies primarily focused on analyzing the global existence and boundedness of solutions \cite{Panday2018Stability,Han2020Cross}, diffusion-induced instability \cite{Wu2018Dynamics}, persistence and coexistence \cite{Liu2021Influence}, Turing instability and Hopf bifurcation \cite{Zhang2023Dynamics}, and other related aspects. By incorporating the concept of fear into mathematical models, these researchers aimed to gain a better understanding of the dynamics and stability of predator-prey systems.

In addition to the fear effect and predator-prey interactions, researchers have also started considering the impact of external environmental forces on the survival and population dynamics of species. Examples of such forces include rivers, wind, and gravity. These external factors can significantly affect the movement and dispersal of individuals in a population, which in turn influences the dynamics and behavior of the population as a whole. Inspired by these real-world scenarios, several scholars have directed their attention towards studying population dynamics in advective environments. This field of research aims to understand how the interaction between species and external forces shapes population dynamics and distribution patterns on
\cite{Jin2014Seasonal,Lutscher2006Effects,Lutscher2007Spatial}.
Speirs and Gurney \cite{Speirs2001Population} investigated a single population reaction-diffusion-advection model:
\begin{eqnarray}\label{aa1}
\begin{cases}
u_t=d_1u_{xx}-q_1u_{x}+u\Big(r-au\Big),\quad t>0, x\in(0,L),\\
d_1u_{x}(0,t)-q_1u(0,t)=0,\quad t>0,\\
u(L,t)=0,\quad t>0.
\end{cases}
\end{eqnarray}
They imposed a no-flux boundary condition on the upstream boundary reflecting the fact that some individuals in the population do not flow upstream, in the river system for instance. On the other hand, individuals that flow downstream from the river into other environments, such as the ocean, may not return. As a result, the downstream boundary condition is set to zero.

Based on the model (\ref{aa1}), some researchers have found that when considering the advection term in the reaction-diffusion model with time delay, there is a possibility of observing Hopf bifurcation under certain boundary conditions \cite{Chen2018Hopf,Chen2020Bifurcation,Bin2021Bifurcation}. Chen et al. \cite{Chen2018Hopf} discovered that the presence of time delay can destabilize spatial nonhomogeneous positive steady-state solutions and lead to oscillation patterns through the occurrence of Hopf bifurcation. Additionally, increasing the advective coefficient increases the likelihood of Hopf bifurcation happening. In another study, Chen et al. \cite{Chen2020Bifurcation}  investigated the dynamic properties of the reaction-diffusion-advection equation with nonlocal terms and explored the existence of nonconstant positive steady-state solutions and the characteristics of Hopf bifurcation.
In the existing literature,
many researchers have extensively studied the stability and bifurcation of predator-prey models from various perspectives and have also explored bifurcation control. However, to our best knowledge, it seems that there is still relatively limited research on diffusion-convection predator-prey models with time delays and fear effects. On the other hand, the functional response functions that describe the interactions between populations hold certain similarities, making it practically significant to abstract and summarize them. Building on this knowledge, we are going to investigate the following predator-prey models.
\begin{eqnarray}\label{523a}
\begin{cases}
    u_t=d_{1}u_{xx} -q_{1}v_x +u\big[r_{0}f(K,v)-d-au(x,t-\tau)\big]-g(u)v, \quad t>0, x\in(0,l\pi), \\
    v_t=d_{2}v_{xx} -q_{2}v_x +v\big(-r_{2}-mv+c g(u)\big), \quad t>0, x\in(0,l\pi),
\end{cases}
\end{eqnarray}
where $u(x,t)$ represents population density at location $x$ and time $t$,
$x=0$ and $x=L$ denote the upstream and downstream ends of the habitat, respectively.
The intrinsic growth rate and maximum environmental capacity of the population are denoted by $r$ and $K$, with the diffusion coefficient and advection rate are represented by $d_1$ and $q_1$, respectively.
where $\tau$ represents the pregnancy time of the prey population.
As we know that there are different types of functional responses $g(u)$ are used to model the interactions between predators and their prey in ecological systems, such as Holling functionals $g(u)=p_0u$ (Type I) and $g(u)=\frac{p_0u}{1+c_1u}$ (Type II) can help us understand how changes in prey density can affect predator populations, and vice versa.
In this study, we try to keep the form $g$ in a general form allowing for a broader range of possibilities beyond the specific Holling functional responses.
The function $f(K,v)$ depends on the level of fear $K$ and predator density $v$, and satisfies
\begin{eqnarray}\label{e718}
\begin{cases}
f(0,v)=1,~{\lim_{K \to \infty}} f(K,v)=0,~f'_{K}(K,v)<0,\\
f(K,0)=1,~~{\lim_{v \to \infty}} f(K,v)=0,~f'_{v}(K,v)<0.
\end{cases}
\end{eqnarray}
In advective environments, where populations experience significant differences in survival conditions, there are typically three types of boundary conditions that have practical application value \cite{Lou2014Evolution}. These include:

(I) Constant-Flux (CF) boundary condition:
\begin{eqnarray*}
d_{1} u_x(x,t)-q_{1} u(x,t)=-q_{1}\kappa_1,~x=0, l\pi, t\geq 0,\\
d_{2} v_x(x,t)- q_{2} v(x,t)=-q_{2}\kappa_2,~x=0, l\pi, t\geq 0.
\end{eqnarray*}
Case (I) applies to organisms in water, where gravity pulls algae towards the bottom (advection), while buoyancy allows organisms to move upwards (diffusion).

(II) Free-Flow (FF) boundary condition:
\begin{eqnarray}\label{a818}
u_x(x,t)=v_x(x,t)=0,~x=0, l\pi, t\geq 0,
\end{eqnarray}
which is commonly used to study how populations can freely pass through boundaries when rivers merge into lakes, and the flux at the boundaries is caused by advection. 

(III) Dirichlet type (D) boundary condition:
\begin{eqnarray}
u(x,t)=\kappa_1,v(x,t)=\kappa_2,~x=0, l\pi, t\geq 0,
\end{eqnarray}
implying that the population density of both upstream and downstream is constant. A common scenario is $\kappa_1=\kappa_2=0$, which means that most freshwater organisms cannot survive in the ocean and will die upon arrival.

Studying the dynamic properties of a predator-prey model with time delay specifically in a river setting holds significant practical application value. Investigating population dynamics in this context allows us to gain insights that can be applied to real-world scenarios and better address practical problems observed in nature. Although there is limited research on the reaction-diffusion-convection equation from a theoretical perspective, approaching this topic from a bifurcation perspective can greatly enhance our understanding of predator-prey models with advection, such as the existence, stability, and bifurcation analysis of positive steady-state solutions for multi-population delayed reaction-diffusion-convection equations \cite{Chen2020Bifurcation,Jin2021Hopf,Li2022Hopf,Liu2023Steady}.

To better understand how movement, dispersal, and connectivity impact population dynamics in advective environments, our discussion will center on the effects of various boundary conditions on the dynamic properties of reaction-diffusion-advection models with a maturation time delay $\tau$. Specifically, we will discuss the model by comparing different boundary conditions to assess their influence on the stability region of the steady-state solution and the fluctuation range of the population. First we assume
that the upstream is CF boundary and the downstream is D boundary, i.e,
\begin{eqnarray}\label{a623}
\begin{cases}
d_{1} u_x(0,t)-q_{1} u(0,t)=-q_{1}\kappa_1,~u(l\pi,t)=\kappa_1, t\geq 0,\\
d_{2} v_x(0,t)- q_{2} v(0,t)=-q_{2}\kappa_2,~v(l\pi,t)=\kappa_2,t\geq 0.
\end{cases}
\end{eqnarray}
Second, we select the FF boundary conditions.
To analyze the existence of a Hopf bifurcation at the positive steady state, we will utilize the time delay $\tau$ as a bifurcation parameter. The center manifold theory and the normal form method will be employed to derive the calculation formula for the normal form of the Hopf bifurcation.
Two models will be chosen as examples to further explain the dynamics phenomena associated with the Hopf bifurcation, one incorporates the Holling-I functional response and fear effect, another one implements the Holling-II functional response.
Furthermore, in the numerical simulations, we will compare the FF/FF boundary conditions with CF/D boundary conditions in (\ref{a623}). This comparison will help us understand how different boundary conditions affect the stability region of the system's steady-state solution as well as the range of population fluctuations due to varying eigenvalues.

The structure of the paper is as follows. In Section \ref{a717}, we obtain the eigenvalue of eigenvalue problem and the adjoint operator under either CF/D boundary conditions or FF/FF boundary conditions. In Section \ref{b717},
our focus lies on studying the linear stability near the positive steady-state solution. We will analyze the direction and stability of the spatial Hopf bifurcation by selecting the generation delay $\tau$ as the bifurcation parameter.
In Section \ref{c717}, we will examine the impact of the linear and Holling-II functional responses on the dynamics of the model under CF/D and FF/FF boundary conditions respectively.

\section{Some preliminaries}\label{a717}
When studying Hopf bifurcations and the stability of periodic solutions using center manifold theory, it is necessary to perform space decomposition.
To make further research more convenient, we introduce definitions and notation,
including Sobolev space and complexification. When the upstream is CF boundary and the
downstream is D boundary,
let
$$Z:=\{(\phi,\psi)\in H^2(0, l\pi)\times H^2(0, l\pi)|d_{1} \phi_{x}(0)-q_1 \phi(0)=0,~\phi(l\pi)=0,
    d_2 \psi_{x}(0)-q_2 \psi(0)=0,~\psi(l\pi)=0\}$$
be the real-valued Sobolev space, and
$$Z_{\mathbb{C}}=Z \oplus iZ=\{x_1+ix_2|x_1, x_2\in Z\}$$
be the complexification of $Z$.
Define $\mathcal{C}=C([-1,0],H^2((0,l\pi)))$, with the inner product:
$$\langle \tilde{a}, \tilde{b}\rangle=\int_{0}^{l\pi}(\bar{a}_1b_1+\bar{a}_2b_2)dx, \forall \tilde{a}(x)=(a_1(x),a_2(x)),\tilde{b}(x)=(b_1(x),b_2(x))\in H^2((0, l\pi)).$$
The norm induced by the inner product is denoted as $\|\cdot\|_{2}$,
$Dom(T)$ denotes the domain of a linear operator $T$. We first quote the following results in \cite{Kubrusly2012Spectral}.
Regarding the eigenvalue problem:
\begin{equation}\label{a25}
\begin{cases}
 d_1\frac{\partial^2X(x)}{\partial
x^2}-q_1\frac{\partial X(x)}{\partial x}=-\nu X(x),\\
d_1\frac{\partial X(0)}{\partial x}-q_1 X(0)=0,~~ X(l\pi)=0,
\end{cases}
\end{equation}
\begin{lemma}\label{c718}
The eigenvalue problem (\ref{a25}) corresponding to the boundary condition given in (\ref{a623}) with $\kappa_1=\kappa_2=0$ has a list of eigenvalues $\{\nu_n\}_{n=1}^{\infty}$ with
$$\nu_{n}=\frac{q_1^2}{4d_1}+d_1\sigma_{n}^2,$$
where $\sigma_{n}$ satisfies $$\tan(l\pi\sigma_{n})=-\frac{2d_1\sigma_{n}}{q_1}.$$
\end{lemma}
%
%
\begin{remark}
It is easy to see that, for $\{\nu_n\}_{n=1}^{\infty}$, $\frac{q_1^2}{4d_1}<\nu_1<\nu_2<\ldots<\nu_n<\nu_{n+1}<\ldots$, and
$\lim\limits_{n\rightarrow\infty}\nu_n=+\infty.$
The standardized eigenfunctions are
\begin{equation}\label{bk}
b_n(x)=\frac{X_n(x)}{\|X_n(x)\|_{2}},n=1,2,\cdots,
\end{equation}
with $\|X_n(x)\|_{2}=\bigg(\int_{0}^{l\pi}X_n^{2}(x)dx\bigg)^{\frac{1}{2}}$.
\end{remark}
We denote the operator $T$ as $T\varphi=d_1\varphi_{xx}-q_1\varphi_{x}$, and
the adjoint operator $T^{\ast}$: $T^{\ast}\phi=d_1\phi_{xx}+q_1\phi_{x}$. Then we have the following result about the domain of $T$ and $T^{\ast}$.
\begin{lemma}\label{d718}
\begin{eqnarray*}
&&Dom(T)=\{\varphi\in H^2((0, l\pi)):d_{1} \varphi_{x}(0)-q_1 \varphi(0)=0,~\varphi(l\pi)=0\},\\
&&Dom(T^{\ast})=\{\phi\in H^2((0, l\pi)):\phi_{x}(0)=0,~\phi(l\pi)=0\}.
\end{eqnarray*}
\end{lemma}

If the boundary condition is FF/FF, which is paralleled to that done in \cite{Duan2019Hopf}. In the additional influence of advection terms, consider the following eigenvalue problem:
\begin{equation}\label{a622}
\begin{cases}
d_1X''(x)-q_1X'(x)=-\nu X(x),~x\in(0,l\pi),\\
X'(x)=0,~~ x=0,l\pi.
\end{cases}
\end{equation}
\begin{lemma}
The eigenvalue problem (\ref{a622}) corresponding to the boundary condition given in (\ref{a818}) has a list of eigenvalues $\{\mu_n\}_{n=0}^{\infty}$ with
\begin{equation}\label{b622}
\mu_{n}=
\begin{cases}
0, ~n=0,\\
\frac{q_1^2}{4d_1}+d_1\frac{n^2}{l^2},  ~n=1,2,\cdots,
\end{cases}
\end{equation}
and $\lim\limits_{n\rightarrow\infty}\mu_n=+\infty$.
The standardized eigenfunctions are
\begin{equation}\label{c622}
c_n(x)=\frac{X_n(x)}{\|X_n(x)\|_{2}},n=0,1,2,\cdots,
\end{equation}
and the characteristic function corresponding to each $\mu_{n}$ is
\begin{eqnarray}\label{d622}
X_n(x)=
\begin{cases}
1, ~n=0,\\
e^{\frac{q_1}{2d_1}x}\bigg(\cos\frac{\sqrt{4d_1\mu_n-q_1^2}}{2d_1}x
-\frac{q_1}{\sqrt{4d_1\mu_n-q_1^2}}
\sin\frac{\sqrt{4d_1\mu_n-q_1^2}}{2d_1}x\bigg),~n=1,2,\ldots.
\end{cases}
\end{eqnarray}
\end{lemma}
\begin{lemma}
\begin{eqnarray*}
&&Dom(T)=\{\varphi\in H^2((0, l\pi)): \varphi_{x}(0)=0,~\varphi_{x}(l\pi)=0\},\\
&&Dom(T^{\ast})=\{\phi\in H^2((0, l\pi)):d_1\phi_{x}(0)+q_1\phi(0)=0,~d_1\phi_{x}(l\pi)+q_1\phi(l\pi)=0\}.
\end{eqnarray*}
\end{lemma}

\section{Linear analysis}\label{b717}
Suppose the system (\ref{523a}) has a unique positive steady state solution $(u_{\ast},v_{\ast})$.
Let
$\tilde{u}=u-u_{\ast}$, $\tilde{v}=v-v_{\ast}$, through Taylor expansion at $(u_{\ast},v_{\ast})$, and drop the tildes in $\tilde{u}$ and $\tilde{v}$,
the system (\ref{523a}) with CF/D boundary conditions becomes
\begin{eqnarray}\label{105c}
\begin{cases}
\frac{\partial u}{\partial t}=d_{1}\frac{\partial^2 u}{\partial x^2}-q_{1}\frac{\partial u}{\partial x}
+\beta_{0}u+\beta_{1}v+\beta_{2}u(x,t-\tau)
+\beta_{3}u^{2}+\beta_{4}uv\\
~~~~~~~+\beta_{5}v^{2}+\beta_{6}uu(x,t-\tau)+\beta_{7}u^{3}+\beta_{8}v^{3}+\beta_{9}u^{2}v+\beta_{10}uv^{2}+..., ~~t>0, x\in(0,l\pi),\\
\frac{\partial v}{\partial t}=d_{2}\frac{\partial^2 v}{\partial x^2} -q_{2}\frac{\partial v}{\partial x}
 +\gamma_{0}u+\gamma_{1}v+\gamma_{2}u^2+\gamma_{3}uv+\gamma_{4}v^2+\gamma_{5}u^3+\gamma_{6}u^2v+..., ~~t>0, x\in(0,l\pi),\\
d_{1}\frac{\partial u(0,t)}{\partial x}-q_{1} u(0,t)=0,~u(l\pi,t)=0, t\geq 0,\\
d_{2}\frac{\partial v(0,t)}{\partial x}-q_{2} v(0,t)=0,~v(l\pi,t)=0,t\geq 0,
\end{cases}
\end{eqnarray}
where
\begin{eqnarray}\label{b213}
\begin{split}
\beta_{0}&=v_{\ast}\bigg[\frac{g(u_{\ast})}{u_{\ast}}-g'(u_{\ast})\bigg],~
\beta_{1}=r_{0}u_{\ast}f'_{v}(K,v_{\ast})-g(u_{\ast}),~\beta_{2}=-au_{\ast},\\
\beta_{3}&=-\frac{1}{2}g''(u_{\ast})v_{\ast},
\beta_{4}=r_{0}f'(K,v_{\ast})-g'(u_{\ast})
,~\beta_{5}=\frac{1}{2}r_{0}u_{\ast}f''_{vv}(K,v_{\ast}),\\
\beta_{6}&=-a,~\beta_{7}=-\frac{1}{6}g'''(u_{\ast})v_{\ast}, ~\beta_{8}=\frac{1}{6}r_0f'''_{vvv}(K,v_{\ast})u_{\ast},~
\beta_{9}=-\frac{1}{2}g''(u_{\ast}), \\
\beta_{10}&=\frac{1}{2}r_0f''_{vv}(K,v_{\ast}),~
\gamma_{0}=cv_{\ast}g'(u_{\ast}),~
\gamma_{1}=-mv_{\ast},~\gamma_{2}=\frac{1}{2}cg''(u_{\ast})v_{\ast},\\
\gamma_{3}&=cg'(u_{\ast}),~\gamma_{4}=-m,~\gamma_{5}=\frac{1}{6}cg'''(u_{\ast})v_{\ast}, ~\gamma_{6}=\frac{1}{2}cg''(u_{\ast}).
\end{split}
\end{eqnarray}
The linearized equation in (\ref{105c}) at $(0,0)$ is
\begin{eqnarray}\label{105d}
\begin{cases}
\frac{\partial u}{\partial t}=d_{1}\frac{\partial^2 u}{\partial x^2}-q_{1}\frac{\partial u}{\partial x}
+\beta_{0}u+\beta_{1}v+\beta_{2}u(x,t-\tau), \quad t>0, x\in(0,l\pi),\\
\frac{\partial v}{\partial t}=d_{2}\frac{\partial^2 v}{\partial x^2} -q_{2}\frac{\partial v}{\partial x}
 +\gamma_{0}u+\gamma_{1}v,\quad t>0, x\in(0,l\pi).
\end{cases}
\end{eqnarray}
To simplify the mathematical analysis, we assume that the diffusion coefficient and advection transmission speed of predator and prey are proportional respectively,
that is,
$d_2=\varepsilon d_1$, $q_2=\varepsilon q_1$.

Let
$$ \chi_{n_1}=\begin{pmatrix} b_n(x)   \\  0  \end{pmatrix},~~~\chi_{n_2}=\begin{pmatrix}   0 \\  b_n(x)  \end{pmatrix},~
n\in \mathbb{N}_{+},$$
where $b_n(x)$ is given in (\ref{bk}).
Obviously, $(\chi_{n_1},\chi_{n_2})$ forms a set of basis in the phase space $X$, and any element $y$ in $X$ can be expanded into the Fourier sequence:
\begin{eqnarray*}
y=\sum_{n=1}^\infty Y^{T}_n \begin{pmatrix} \chi_{n_1} \\  \chi_{n2} \end{pmatrix},
Y^{T}_n=(\langle y,\chi_{n_1} \rangle,\langle y,\chi_{n_2} \rangle).
\end{eqnarray*}
The characteristic equation is
\begin{eqnarray}\label{206a}
\sum_{n=1}^\infty Y^{T}_n \bigg[\lambda I_2+\begin{pmatrix}\nu_n-\beta_0-\beta_2e^{-\lambda \tau} & -\beta_1 \\ -\gamma_0  &  \varepsilon\nu_n-\gamma_1 \end{pmatrix} \bigg] \begin{pmatrix} \chi_{n_1} \\  \chi_{n_2} \end{pmatrix}=0,
\end{eqnarray}
which is equivalent to
\begin{eqnarray}\label{206b}
\lambda^{2}+M_n\lambda+T_n+\beta_2e^{-\lambda \tau}(-\lambda-\varepsilon\nu_n+\gamma_1)=0,
\end{eqnarray}
with $2\times2$ identity matrix $I_2$ and $M_n=(\varepsilon+1)\nu_n-\beta_0-\gamma_1$, $T_n=(\nu_n-\beta_0)(\varepsilon\nu_n-\gamma_1)-\gamma_0\beta_1$.
When $\tau=0$,
it is easy to obtain that
\begin{lemma}\label{b212}
If $(S_1)$ holds, the positive steady state solution $(u_{\ast},v_{\ast})$ of (\ref{523a}) is asymptotically stable with $\tau=0$, where
$$(S_1)~~\beta_2-M_n<0, \ \text{and} \ ~C_n=T_n+\beta_2(\gamma_1-\varepsilon\nu_n)>0.$$
\end{lemma}

When $\tau>0$, let $\lambda=i\omega$ $(\omega>0)$ be a root of the characteristic equation (\ref{206b}). Then
$\omega$ satisfies the following equation
\begin{eqnarray}\label{210a}
-\omega^2+M_ni\omega+T_n+\beta_2(\cos \omega\tau-i\sin\omega\tau)(-i\omega-\varepsilon\nu_n+\gamma_1)=0.
\end{eqnarray}
By separating the real and imaginary parts of (\ref{210a}), we get
\begin{equation}\label{210b}
\begin{cases}
\omega^2-T_n=\beta_2\big[(\gamma_1-\varepsilon\nu_n)\cos \omega\tau-\omega\sin \omega\tau\big],\\
-M_n\omega=\beta_2\big[-\omega\cos \omega\tau-(\gamma_1-\varepsilon\nu_n)\sin\omega\tau\big].
\end{cases}
\end{equation}
Adding the squares of the two equations in (\ref{210b}) and denote
$z=\omega^{2}$, $D_n=M_n^{2}-2T_n-\beta_2^2$, $E_n=T_n-\beta_2(\gamma_1-\varepsilon\nu_n)$, and $C_n$ in Lemma \ref{b212}, one have
\begin{equation}\label{211b}
z^{2}+D_nz+C_nE_n=0,
\end{equation}
If there exists an certain integer $n_0=\{1,2,...\}$, such that
$$(S_2)~~E_{n_0}<0,$$
or $$(S_3)~~D_{n_0}<0,E_{n_0}>0, D_{n_0}^2-4C_{n_0}E_{n_0}>0,$$
then
\begin{eqnarray*}
z_{n_0}=\begin{cases}
\frac{1}{2}\big[-D_{n_0}+\sqrt{D_{n_0}^2-4C_{n_0}E_{n_0}}\big], \mathrm{under}~(S_1)~\mathrm{and}~(S_2),  \\
\frac{1}{2}\big[-D_{n_0}\pm\sqrt{D_{n_0}^2-4C_{n_0}E_{n_0}}\big], \mathrm{under}~(S_1)~\mathrm{and}~(S_3).
\end{cases}
\end{eqnarray*}
From (\ref{210b}), we have
\begin{equation}\label{211c}
\begin{cases}
\cos \omega\tau=\frac{(\omega^2-T_n)(\gamma_1-\varepsilon\nu_n)+M_n\omega^2}{\beta_2[(\gamma_1-\varepsilon\nu_n)^2+\omega^2]}\triangleq C_n(\omega),\\
\sin\omega\tau=\frac{M_n\omega(\gamma_1-\varepsilon\nu_n)-\omega(\omega^2-T_n)}{\beta_2[(\gamma_1-\varepsilon\nu_n)^2+\omega^2]}\triangleq S_n(\omega).
\end{cases}
\end{equation}
Define
\begin{equation}\label{a819}
\tau_{n_0}^{j\pm}=\begin{cases}
\frac{1}{\omega_{n_0}^{\pm}}\big(\arccos (C_n(\omega_{n_0}^{\pm}))+2j\pi\big),~~ \mathrm{if}~ S_n(\omega_{n_0}^{\pm}\big)\geq 0,\\
\frac{1}{\omega_{n_0}^{\pm}}\big(2\pi-\arccos (C_n(\omega_{n_0}^{\pm}))+2j\pi\big), ~~\mathrm{if}~ S_n(\omega_{n_0}^{\pm})< 0,
\end{cases}
j\in\{0,1,2,...\}.
\end{equation}
Let $\lambda(\tau)=\alpha(\tau)+i\beta(\tau)$ be the root of (\ref{206b}) near $\tau=\tau_{n_0}^{j\pm}$ satisfying
\begin{eqnarray}\label{a718}
\alpha(\tau_{n_0}^{j\pm})=0,~~\beta(\tau_{n_0}^{j\pm})=\omega_{n_0}^{\pm},~~j=0,1,2,...,
\end{eqnarray}
then we have the following transversality condition.
\begin{lemma}\label{a212}
Suppose that $(S_1)$ and $(S_3)$ are satisfied. Then
$$\mathrm{sign} \bigg\{\mathrm{Re}\Big(\frac{d\lambda}{d\tau}\Big) \bigg\}_{\tau=\tau_{n_0}^{j+}}>0,~\mathrm{sign} \bigg\{\mathrm{Re}\Big(\frac{d\lambda}{d\tau}\Big) \bigg\}_{\tau=\tau_{n_0}^{j-}}<0.$$
\end{lemma}
\begin{proof}
Differentiating the two sides of (\ref{206b}) with respect to $\tau$ leads to
$$\bigg(\frac{d\lambda}{d\tau}\bigg)^{-1}=\frac{2\lambda+M_n-\beta_2e^{-\lambda \tau}}
{\lambda\beta_2e^{-\lambda \tau}(-\lambda-\varepsilon\nu_n+\gamma_1)}-\frac{\tau}{\lambda},$$
thus
\begin{eqnarray*}
\mathrm{Re} \bigg(\frac{d\lambda}{d\tau}\bigg)^{-1}&=&\mathrm{Re} \bigg[\frac{2\lambda+M_n}{-\lambda(\lambda^2+M_n\lambda+T_n)}\bigg]-\mathrm{Re} \bigg[\frac{1}{\lambda(-\lambda-\varepsilon\nu_n+\gamma_1)}\bigg],\\
&=&\frac{\pm\sqrt{D_n^2-4C_nE_n}}{\beta_2^2[\omega_{n_0}^2+(\gamma_1-\varepsilon\nu_n)^2]}.
\end{eqnarray*}
Finally, we can get
$$\mathrm{sign} \bigg\{\mathrm{Re}\Big(\frac{d\lambda}{d\tau}\Big) \bigg\}_{\tau=\tau_{n_0}^{j+}}
=\mathrm{sign} \Big\{\sqrt{D_n^2-4C_nE_n}\Big\}>0,$$
$$\mathrm{sign} \bigg\{\mathrm{Re}\Big(\frac{d\lambda}{d\tau}\Big) \bigg\}_{\tau=\tau_{n_0}^{j-}}
=\mathrm{sign} \Big\{-\sqrt{D_n^2-4C_nE_n}\Big\}<0.$$
\end{proof}

\begin{remark}
Suppose that $(S_1)$ and $(S_2)$ are satisfied. Then
$\mathrm{sign} \bigg\{\mathrm{Re}\Big(\frac{d\lambda}{d\tau}\Big) \bigg\}_{\tau=\tau_{n_0}^{j+}}>0$.
\end{remark}

Denote $\bar{\tau}=\mathrm{min}\{\tau_{n_0}^{0+}\}$, $n_0\in\{1,2,3,...\}$ if $(S_2)$ holds, or
$\tilde{\tau}=\mathrm{min}\{\tau_{n_0}^{0+},\tau_{n_0}^{0-}\}$, $n_0\in\{1,2,3,...\}$ if $(S_3)$ holds. The corresponding purely imaginary roots are $\pm i\omega_{n_0}^{+}$ and $\pm i\omega_{n_0}^{\pm}$, respectively.
Applying Lemmas \ref{b212} and \ref{a212}, we can get the following conclusions.
\begin{theorem}\label{b819}
Assume that $(S_1)$ holds. For the system (\ref{523a}), we have\\
(i) If $(S_2)$ holds, then the positive steady state solution $(u_{\ast},v_{\ast})$ is asymptotically stable for $\tau\in[0,\bar{\tau})$; $(u_{\ast},v_{\ast})$  is unstable for $\tau>\bar{\tau}$;
$\tau=\tau_{n_0}^{j+}$ $(j = 0, 1, 2, ...)$ are Hopf bifurcation values and these Hopf bifurcations are all spatially inhomogeneous.\\
(ii) If $(S_3)$ holds, there are finite critical points $\tilde{\tau}<\tau^{1}<...<\tau^{k}$, and $\tau^{1},\tau^{2},...,\tau^{k}\in \{\tau_{n_0}^{j\pm}\}$, $k,j\in \mathbb{N}_0$. When $\tau\in[0,\tilde{\tau})\cup(\tau^{1},\tau^{2})\cup...\cup(\tau^{k-1},\tau^{k})$, $(u_{\ast},v_{\ast})$ is asymptotically stable; When $\tau\in(\tilde{\tau},\tau^{1})\cup(\tau^{2},\tau^{3})\cup...\cup(\tau^{k},+\infty)$, $(u_{\ast},v_{\ast})$ is unstable.
\end{theorem}

\section{Direction and stability of spatial Hopf bifurcation}\label{c717}
In the previous section, we have obtained some conditions under which a family of spatially inhomogeneous periodic solutions bifurcate from the positive constant steady state $(u_{\ast},v_{\ast})$  of (\ref{523a}) when $\tau$ crosses the critical value $\tau_{n_0}^{j\pm}$. In this section, we will investigate the directions and the stability of bifurcating periodic solutions by using the center manifold theorem and the normal formal theory of partial functional differential equations \cite{Wu1996Theory,Faria2002Trans}.

Let $\tilde{u}(x,t)=u(x,\tau t), \tilde{v}(x,t)=v(x,\tau t)$, $\tau=\tilde{\tau}+\mu$, $\mu\in \mathbb{R}$, and drop the tilde for convenience.
Then in the space $\mathcal{C}:=C([-1,0],X)$,
(\ref{105c}) can be written as an abstract form
\begin{equation}\label{a1}
    \frac{d}{dt}U(t)=\tilde{\tau}D\Delta U(t)-\tilde{\tau}M\nabla U(t)+L({\tilde{\tau}})(U_t)+F(\mu,U_t),
\end{equation}
where $U=(u_1,u_2)^{T}$, $u_1(t)=u(\cdot,t)$, $u_2(t)=v(\cdot,t)$, and
$$F(\mu,\phi)=\mu D\Delta \phi(0)-\mu M\nabla \phi(0)+L(\mu)(\phi)+f(\mu,\phi),\phi=(\phi_1,\phi_2)^{T}\in\mathcal{C},$$
$$L(\mu)(\cdot):\mathcal{C}\rightarrow X, F:\mathcal{C}\times \mathbb{R}\rightarrow X,
D=\begin{pmatrix} d_1 & 0  \\  0 & d_2 \end{pmatrix},~~M=\begin{pmatrix} q_1 & 0  \\  0 & q_2 \end{pmatrix},
$$
$$L(\tilde{\tau})(\phi)=\tilde{\tau}\begin{pmatrix} \beta_0\phi_{1}(0)+\beta_1\phi_{2}(0)+\beta_2\phi_{1}(-1) \\ \gamma_0\phi_{1}(0)+\gamma_1\phi_{2}(0) \end{pmatrix},L(\mu)(\phi)=\mu\begin{pmatrix} \beta_0\phi_{1}(0)+\beta_1\phi_{2}(0)+\beta_2\phi_{1}(-1) \\ \gamma_0\phi_{1}(0)+\gamma_1\phi_{2}(0) \end{pmatrix},$$
\begin{eqnarray*}
&&f(\mu,\phi)=(\tilde{\tau}+\mu)\cdot\begin{pmatrix} \beta_{3}\phi_{1}^2(0)+\beta_{4}\phi_{1}(0)\phi_{2}(0)+
\beta_{5}\phi_{2}^2(0)+\beta_{6}\phi_{1}(0)\phi_{1}(-1)\\+
\beta_{7}\phi_{1}^3(0)+\beta_{8}\phi_{2}^3(0)+\beta_{9}\phi_{1}^2(0)\phi_{2}(0)+
\beta_{10}\phi_{1}(0)\phi_{2}^2(0)+... \\
\gamma_{2}\phi_{1}^2(0)+\gamma_{3}\phi_{1}(0)\phi_{2}(0)+\gamma_{4}\phi_{2}^2(0)
+\gamma_{5}\phi_{1}^3(0)+\gamma_{6}\phi_{1}^2(0)\phi_{2}(0)+....
\end{pmatrix}
\end{eqnarray*}
For the linearized system
\begin{equation}\label{a415}
\frac{d}{dt}U(t)=\tilde{\tau}D\Delta U(t)-\tilde{\tau}M\nabla U(t)+L({\tilde{\tau}})(U_t),
\end{equation}
from (\ref{a718}) we know that $\{i\omega_{n_0}\tilde{\tau},-i\omega_{n_0}\tilde{\tau}\}$ is a pair of simple purely imaginary eigenvalues of the characteristic equation.
By the Riesz representation theorem, for $\phi\in C([-1,0],\mathbb{R}^2)$, there exists a $2\times2$ matrix function $\eta^{k}(\theta,\tilde{\tau})$ $(\theta\in[-1,0])$ such that
\begin{eqnarray*}
-\tilde{\tau}\nu_n\begin{pmatrix} 1 & 0\\ 0 & \varepsilon \end{pmatrix}\phi(0)+L({\tilde{\tau}})(\phi)=\int_{-1}^{0}d\eta^k(\theta,\tilde{\tau})\phi(\theta),
\end{eqnarray*}
where
\begin{eqnarray}
\eta^k(\theta,\tilde{\tau})=
\begin{cases}
\tilde{\tau}\begin{pmatrix} -\nu_n+\beta_0 & \beta_1 \\ \gamma_0 & -\varepsilon\nu_n+\gamma_1 \end{pmatrix}, & \theta=0,\\
0, & \theta\in(-1,0),\\
\tilde{\tau} \begin{pmatrix} -\beta_2 & 0\\ 0 & 0 \end{pmatrix}, & \theta=-1.
\end{cases}
\end{eqnarray}
Let $A(\tilde{\tau})$ be the infinitesimal generators of the semigroup induced by the solutions of (\ref{a415}) and $A^\ast$ denote the formal adjoint of $A(\tilde{\tau})$ under the following bilinear pairing
\begin{equation}\label{bi}
    (\psi,\phi)=\psi(0)\phi(0)-\int_{-1}^{0}\int_{\xi=0}^{\theta}\psi(\xi-\theta)d\eta^k(\theta, \tilde{\tau})\phi(\xi)d\xi, \quad k=1,2,\cdots,
\end{equation}
where $\phi\in C([-1,0],\mathbb{R}^2)$, $\psi\in C([0,1],\mathbb{R}^2)$.
Using the calculation method in \cite{Bin2023}, $A^\ast$ satisfies
\begin{equation}\label{a404}
\begin{cases}
A^\ast m=d_1m_{xx}+q_1 m_{x},\\
m_{x}(0)=0,m(l\pi)=0.
\end{cases}
\end{equation}
From Lemma \ref{c718} and Lemma \ref{d718}, we know that the associated eigenfunction is
\begin{equation}\label{b404}
\tilde{b}_n=\frac{m_n}{\|m_n\|_{2}},
\end{equation}
where
\begin{eqnarray}\label{a602}
m_n=e^{-\frac{q_1}{2d_1}x}\bigg(\cos\frac{\sqrt{4d_1\nu_n-q_1^2}}{2d_1}x
+\frac{q_1}{\sqrt{4d_1\nu_n-q_1^2}}
\sin\frac{\sqrt{4d_1\nu_n-q_1^2}}{2d_1}x\bigg),~n=1,2,\ldots.
\end{eqnarray}
Let $P$ and $P^{\ast}$ be the center subspace, which is the generalized eigenspace of $A(\tilde{\tau})$ and $A^\ast$, respectively. By tidious symbonic calculations, see \cite{Hale1977Theory,Wu1996Theory}, we can give the following results.
\begin{lemma}
Let
\begin{eqnarray*}
&&q_1=\frac{\gamma_0}{i\omega_{n_0}+\varepsilon\nu_{n_0}-\gamma_1},q_2^{\ast}=\frac{\beta_1}{i\omega_{n_0}+\varepsilon\nu_{n_0}-\gamma_1},\\
&&M=\frac{1}{1+q_1q_2^{\ast}+\beta_2\tilde{\tau}e^{-i\omega_{n_0}\tilde{\tau}}}.
\end{eqnarray*}
Then
$$q(\theta)=\begin{pmatrix} 1 \\ q_1  \end{pmatrix}e^{i\omega_{n_0}\tilde{\tau}\theta},\theta\in[-1,0]$$
is a basis of $P$ with $\{i\omega_{n_0}\tilde{\tau},-i\omega_{n_0}\tilde{\tau}\}$;
$$\hat{q}^{\ast}(\theta)=M(1,q^{\ast}_2)e^{-i\omega_{n_0}\tilde{\tau}s},\theta\in[0,1].$$
is a basis of $P^{\ast}$ with $\{i\omega_{n_0}\tilde{\tau},-i\omega_{n_0}\tilde{\tau}\}$.
\end{lemma}

By the normal form method in \cite{Faria2002Trans}, we can obtain the following key parameters.
\begin{eqnarray*}\label{c426}
\Gamma_1&=&\frac{i}{2\omega_{0}\tau_{0}}\Big
(\mathbf{K}_{11}\mathbf{K}_{20}-2|\mathbf{K}_{11}|^{2}-\frac{|\mathbf{K}_{02}|^{2}}{3}\Big)+\frac{\mathbf{K}_{21}}{2},\\
\Gamma_2&=&-\frac{\mathrm{Re}(\Gamma_1)}{\mathrm{Re}(\lambda'(\tau_{0}))},\\
\Gamma_3&=&2\mathrm{Re}(\Gamma_1),
\end{eqnarray*}
with
\begin{eqnarray*}\label{c426}
\mathbf{K}_{20}&=&2\tilde{\tau}M(1,q^{\ast}_2)\cdot\begin{pmatrix} \beta_3+\beta_4q_1+\beta_5q^{2}_1+\beta_6e^{-i\omega_{n_0}\tilde{\tau}}
\\ \gamma_2+\gamma_3q_1+\gamma_4q^{2}_1\end{pmatrix}
\int_{0}^{l\pi}b_{n_0}^2\tilde{b}_{n_0}dx,\\
\mathbf{K}_{11}&=&\tilde{\tau}M(1,q^{\ast}_2)\cdot\begin{pmatrix} 2\beta_3+\beta_4(q_1+\bar{q}_1)+2\beta_5q_1\bar{q}_1
+\beta_6(e^{-i\omega_{n_0}\tilde{\tau}}+e^{i\omega_{n_0}\tilde{\tau}})
\\ 2\gamma_2+\gamma_3(q_1+\bar{q}_1)+2\gamma_4q_1\bar{q}_1\end{pmatrix}
\int_{0}^{l\pi}b_{n_0}^2\tilde{b}_{n_0}dx,\\
\mathbf{K}_{21}&=&2\tilde{\tau}M(1,q^{\ast}_2)\cdot\begin{pmatrix} 3\beta_7+3\beta_8q_1^{2}\bar{q}_1+\beta_9(\bar{q}_1+2q_1)+\beta_{10}(q_1^2+2q_1\bar{q}_1)
\\ 3\gamma_5+\gamma_6(\bar{q}_1+2q_1)\end{pmatrix}
\int_{0}^{l\pi}b_{n_0}^3\tilde{b}_{n_0}dx\\
&+&2\tilde{\tau}M\int_{0}^{l\pi}\mathbf{K}b_{n_0}\tilde{b}_{n_0}dx,
\end{eqnarray*}
where
\begin{eqnarray*}
\mathbf{K}&=&\beta_3\big[V_{20}^{(1)}(0)+2V_{11}^{(1)}(0)\big]+\beta_4\Big[V_{11}^{(2)}(0)+
\frac{V_{20}^{(2)}(0)}{2}+\frac{V_{20}^{(1)}(0)}{2}\bar{q}_1+V_{11}^{(1)}(0)q_1\Big]\\
&+&\beta_5\big[V_{20}^{(2)}(0)\bar{q}_1+2V_{11}^{(2)}(0)q_1\big]+\beta_6\bigg[V_{11}^{(1)}(-1)
+\frac{V_{20}^{(1)}(-1)}{2}+
\frac{V_{20}^{(1)}(0)}{2}e^{i\omega_{n_0}\tilde{\tau}}+V_{11}^{(1)}(0)e^{-i\omega_{n_0}\tilde{\tau}}\bigg]\\
&+&q^{\ast}_2\gamma_2\big[V_{20}^{(1)}(0)+2V_{11}^{(1)}(0)\big]+q^{\ast}_2\gamma_3\Big[V_{11}^{(2)}(0)+
\frac{V_{20}^{(2)}(0)}{2}+\frac{V_{20}^{(1)}(0)}{2}\bar{q}_1+V_{11}^{(1)}(0)q_1\Big]\\
&+&q^{\ast}_2\gamma_4\big[V_{20}^{(2)}(0)\bar{q}_1+2V_{11}^{(2)}(0)q_1\big].
\end{eqnarray*}

We can further derive the expressions for $V_{20}(\theta)$ and $V_{11}(\theta)$ $(\theta\in[-1,0])$ as follows.
\begin{eqnarray*}
	&&V_{20}(\theta)=\frac{-\mathbf{K}_{20}}{i\omega_{n_0}\tilde{\tau}}q(0)
	e^{i\omega_{n_0}\tilde{\tau}\theta}b_{n_0}-
	\frac{\bar{\mathbf{K}}_{02}}{3i\omega_{n_0}\tilde{\tau}}\bar{q}(0)
	e^{-i\omega_{0}\tilde{\tau}\theta}b_{n_0}+
	\Lambda_{1}e^{2i\omega_{n_0}\tilde{\tau}\theta},\\
	&&V_{11}(\theta)=\frac{\mathbf{K}_{11}}{i\omega_{n_0}\tilde{\tau}}q(0)
	e^{i\omega_{n_0}\tilde{\tau}\theta}b_{n_0}-
	\frac{\bar{\mathbf{K}}_{11}}{i\omega_{n_0}\tilde{\tau}}\bar{q}(0)
	e^{-i\omega_{n_0}\tilde{\tau}\theta}b_{n_0}+\Lambda_{2},
\end{eqnarray*}
where
\begin{eqnarray*}
	\Lambda_{1}=\sum^{\infty}_{n=1}\Lambda_{1,n}b_n,~ \Lambda_{2}=\sum^{\infty}_{n=1}\Lambda_{2,n}b_n.
\end{eqnarray*}
$\Lambda_{1,n}$ and $\Lambda_{2,n}$ have the following representation.
\begin{eqnarray*}
	\Lambda_{1,n}
	&=&\tilde{\tau}^{-1} L_1^{-1}\cdot2\begin{pmatrix} \beta_3+\beta_4q_1+\beta_5q_1^2+\beta_6e^{-i\omega_{n_0}\tilde{\tau}} \\ \gamma_2+\gamma_3q_1+\gamma_4q_1^2 \end{pmatrix}\int_{0}^{l\pi}b_{n_0}^{2}b_{n}dx,\\
\Lambda_{2,n}
	&=&\tilde{\tau}^{-1} L_2^{-1}\cdot2\begin{pmatrix} 2\beta_3+\beta_4(q_1+\bar{q}_1)+2\beta_5q_1\bar{q}_1
+\beta_6(e^{i\omega_{n_0}\tilde{\tau}}+e^{-i\omega_{n_0}\tilde{\tau}}) \\ 2\gamma_2+\gamma_3(q_1+\bar{q}_1)+2\gamma_4q_1\bar{q}_1 \end{pmatrix}\int_{0}^{l\pi}b_{n_0}^{2}b_{n}dx,
\end{eqnarray*}
with
\begin{eqnarray*}
L_1=\begin{pmatrix} 2i\omega_{n_0}+\nu_{n}-\beta_0-\beta_2e^{-2\mathrm{i}\omega_{n}\tilde{\tau}} & -\beta_1 \\  -\gamma_0 & 2i\omega_{n_0}+\varepsilon \nu_{n_0}-\gamma_1 \end{pmatrix},
L_2=\begin{pmatrix} \nu_{n}-\beta_0-\beta_2 & -\beta_1 \\  -\gamma_0 & \varepsilon \nu_{n}-\gamma_1 \end{pmatrix}.
\end{eqnarray*}

As it is well known that when $\Gamma_2>0 (<0)$, the Hopf bifurcation is forward (backward) 
and the bifurcating periodic solutions are orbitally stable (unstable) if $\Gamma_3< 0 (> 0)$.

\section{Influence of different functional responses and different boundary conditions}
In order to capture the biological significance of the fear effect,
in this section,
by selecting the specific form of $f(K, v)$ as $f(K,v)=\frac{1}{1+Kv}$ satisfying (\ref{e718}), our focus is on investigating the influence of different
functional response function $g(u)$ and different boundary conditions on the dynamics in the interaction in predator and prey.

I. With CF/D boundary

a) When we select the linear functional response function $g(u)=p_0u$,
the system (\ref{523a}) becomes
\begin{eqnarray}\label{607a}
\begin{cases}
    \frac{\partial u}{\partial t}=d_{1}\frac{\partial^2 u}{\partial x^2} -q_{1}\frac{\partial u}{\partial x} +u\big[\frac{r_0}{1+Kv}-d-au(x,t-\tau)\big]-p_0uv, \quad t>0, x\in(0,l\pi), \\
    \frac{\partial v}{\partial t}=d_{2}\frac{\partial^2 v}{\partial x^2} -q_{2}\frac{\partial v}{\partial x} +v\big(-r_{2}-mv\big)+c p_0uv, \quad t>0, x\in(0,l\pi).
\end{cases}
\end{eqnarray}
Let $m_0=\frac{a}{cp_0}>0$, $A_1=K(p_0+mm_0)>0$, $A_2=p_0+Km_0r_2+mm_0+dK>0$, and $A_3=m_0r_2+d-r_0$.

When $A_3<0$, in (\ref{607a}), there is a unique positive steady-state solution $(u_\ast,v_\ast)$ with
$v_\ast=\frac{-A_2+\sqrt{A_2^2-4A_1A_3}}{2A_1}$, $u_\ast=\frac{r_2+mv_\ast}{cp_0}$.
Consequently, the coefficients in (\ref{b213}) have the following explict form,
\begin{eqnarray*}
\begin{split}
\beta_{0}&=0,~\beta_{1}=-\frac{Kr_0u_{\ast}}{(1+Kv_{\ast})^2}-p_0u_{\ast},~\beta_{2}=-au_{\ast},~\beta_{3}=0,~
\beta_{4}=-\frac{Kr_0}{(1+Kv_{\ast})^2}-p_0,\\
\beta_{5}&=\frac{K^2r_0u_{\ast}}{(1+Kv_{\ast})^3},~\beta_{6}=-a,~\beta_{7}=0,
~\beta_{8}=-\frac{K^3r_0u_{\ast}}{(1+Kv_{\ast})^4},~\beta_{9}=0,~\beta_{10}=\frac{K^2r_0}{(1+Kv_{\ast})^3},\\
\gamma_{0}&=cp_0v_{\ast},~\gamma_{1}=-mv_{\ast},~\gamma_{2}=0,~\gamma_{3}=cp_0,~\gamma_{4}=-m,
~\gamma_{5}=0,~\gamma_{6}=0.
\end{split}
\end{eqnarray*}

With the fixed parameters as
\begin{eqnarray}\label{e622}
&&K=10,d=0.04,a=0.06,p_0=0.8,r_2=0.5,c=0.4,\nonumber \\
&&m=0.1,d_{1}=0.3,\varepsilon=8,
l=10,q_1=0.001,
\end{eqnarray}
when we fix $r_0=0.2$, here $\tilde{\tau}=7.3764$, numerically we can obtain the unique positive steady-state solution as $(u_{\ast}, v_{\ast})=(1.5712, 0.0278)$.
Enlarge the range of $r_0$ to $[0.13,1]$, where $(S_1)$ and $(S_3)$ hold.
According to (\ref{a819}), taking $r_0$ as a parameter, the stability region with respect to $\tau$ is given in Fig. \ref{a525}.
From Lemma \ref{a212} and Theorem \ref{b819}(ii), it is easy to obtain from Fig. \ref{a525} that the positive steady-state solution $(u_{\ast},v_{\ast})$ is asymptotically stable when
$\tau\in[0,7.3764)\cup(35.5432,35.7352)$. More clearly, Fig. \ref{a608} shows the local asymptotic stability  of the positive steady state when $\tau=5$. Fig. \ref{b608}(a) shows stable spatially nonhomogeneous periodic solutions when $\tau=8$. (The dynamics of predators are similar to those of prey, which is ignored.)
In addition to the effect of time delay, in Fig. \ref{b608}, we compared the FF/FF boundary condition with the CF/D boundary condition and found that the prey exhibits homogeneous periodic changes.
The boundary conditions have a significant impact on population density changes. In addition, if the advective environment is not considered, readers can refer to \cite{Duan2019Hopf}, and the population also exhibits homogeneous periodic changes.
As the time delay is much greater than $\tilde{\tau}$, taking 190 and 480 respectively, the system (\ref{607a}) exhibits unstable transients and unstable spatially nonhomogeneous quasi-periodic solutions, see Fig. \ref{d608}(a) and Fig. \ref{e608}. Under the FF/FF boundary conditions, when $\tau$ is sufficiently large, we find that the system exhibits unstable spatially homogeneous quasi-periodic solutions. Fig. \ref{d608}(b) shows the variation of the number of prey over time and space. The initial functions are all $u(x,0)=u_{\ast}+0.01\cos x$, $v(x,0)=v_{\ast}+0.01\cos x$ in Figs.\ref{a608}-\ref{e608}.
\begin{figure}[h]
\centering{\includegraphics[width=0.5\textwidth]{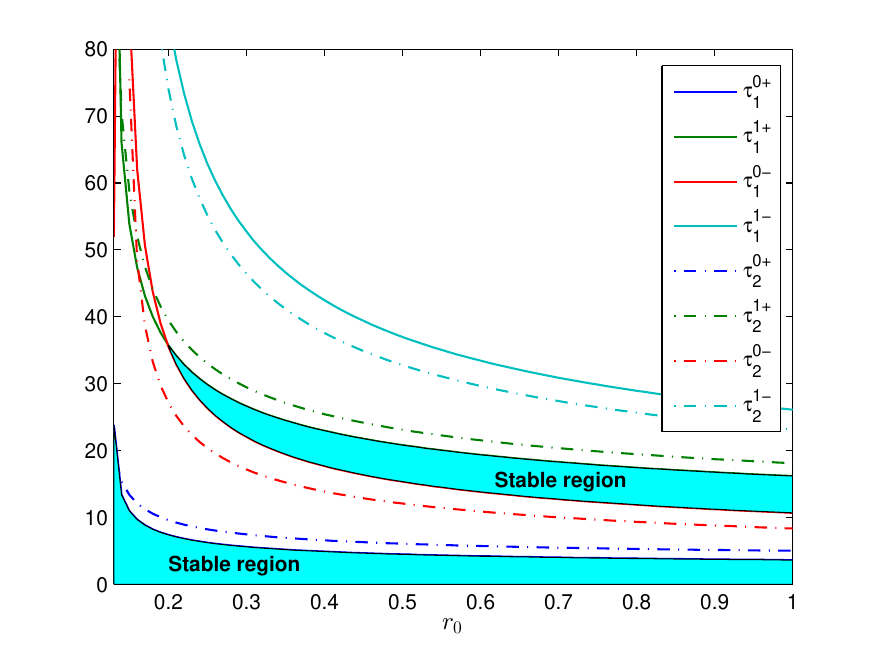}}
\caption{Partial Hopf bifurcation lines and stable regions on the $r_0-\tau$ plane.}
\label{a525}
\end{figure}

\begin{figure}[h]
\centering{\includegraphics[width=0.8\textwidth]{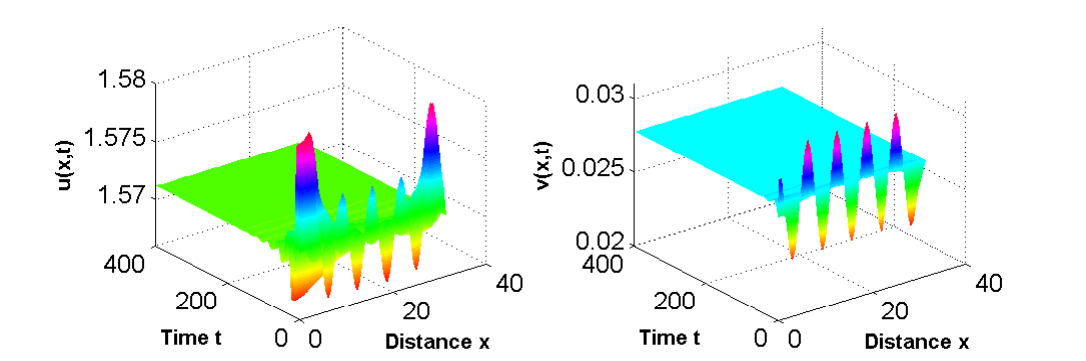}}
\caption{When $\tau=5$, the positive steady state of (\ref{607a}) is locally asymptotically stable.}
\label{a608}
\end{figure}

\begin{figure}[h]
\centering{(a)\includegraphics[width=0.4\textwidth]{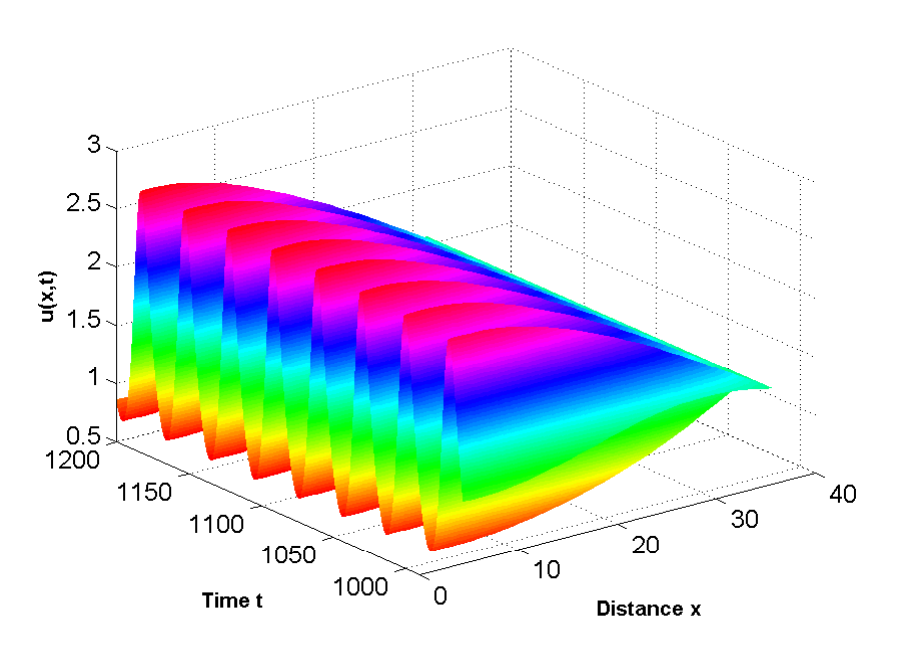}}
(b)\includegraphics[width=0.4\textwidth]{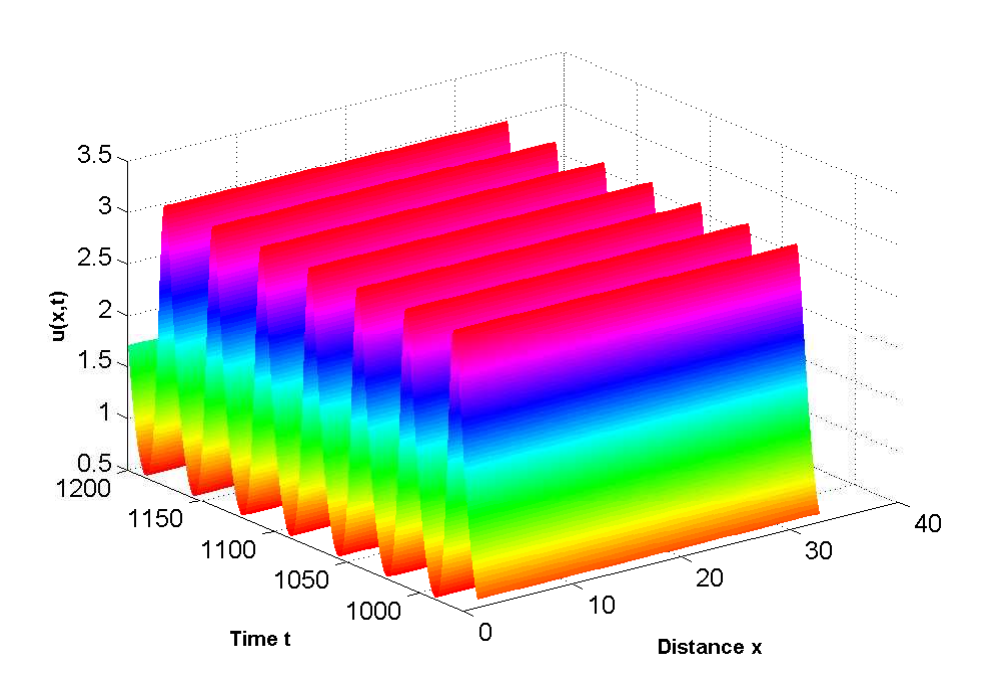}
\caption{When $\tau=8$, the system (\ref{607a}) exhibits stable spatially nonhomogeneous and homogeneous periodic solutions with different boundary conditions, respectively. (a)CF/D boundary; (b)FF/FF boundary.}
\label{b608}
\end{figure}

\begin{figure}[h]
\centering{(a)\includegraphics[width=0.4\textwidth]{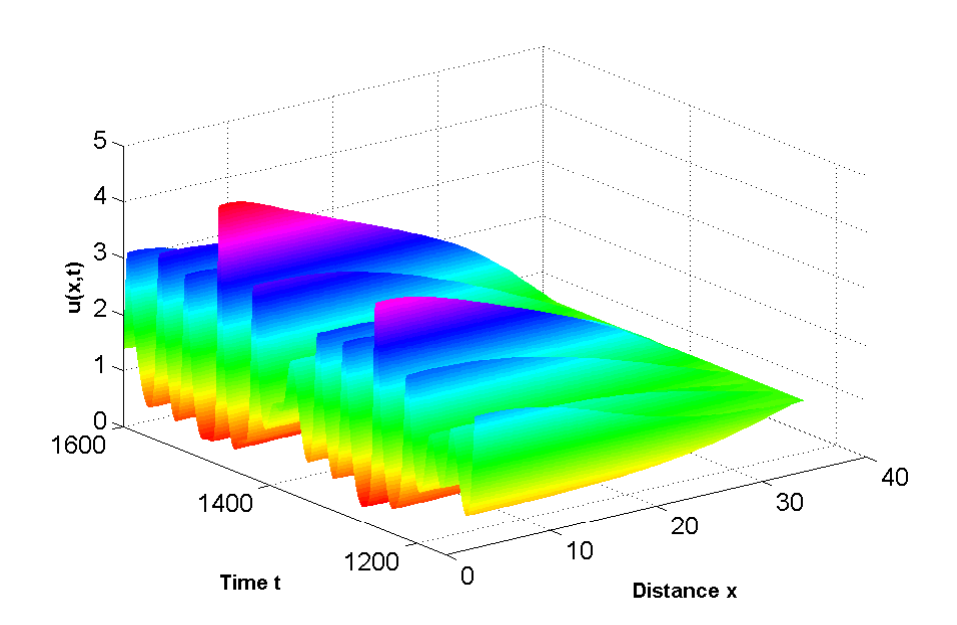}}
(b)\includegraphics[width=0.4\textwidth]{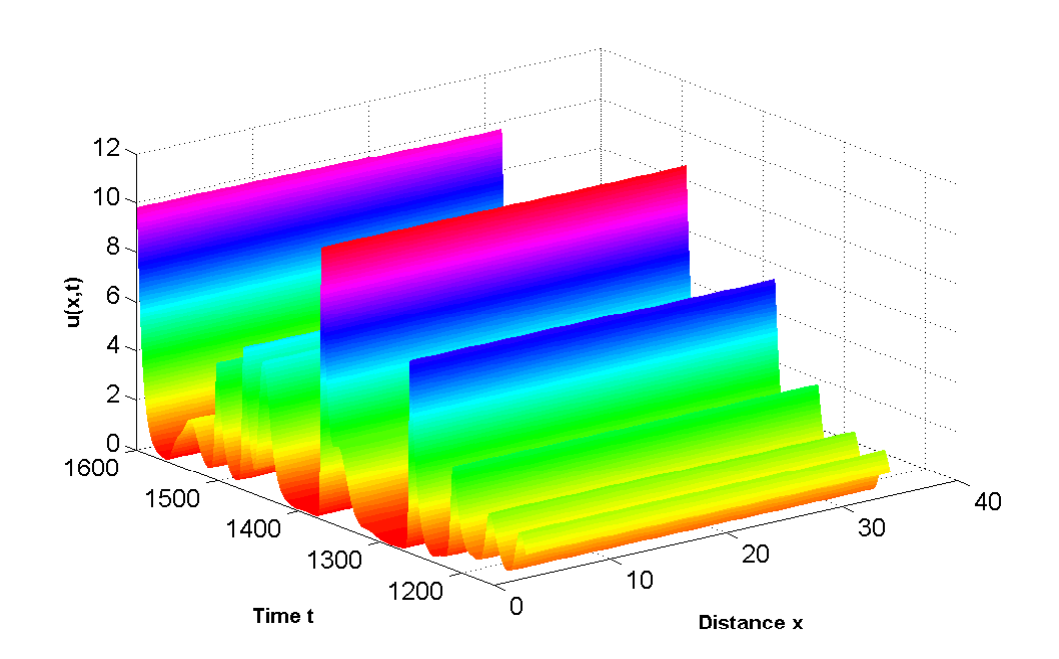}
\caption{When $\tau=190$, the system (\ref{607a}) exhibits unstable spatially nonhomogeneous/homogeneous quasi-periodic solutions. (a)CF/D boundary; (b)FF/FF boundary.}
\label{d608}
\end{figure}

\begin{figure}[h]
\centering{\includegraphics[width=0.8\textwidth]{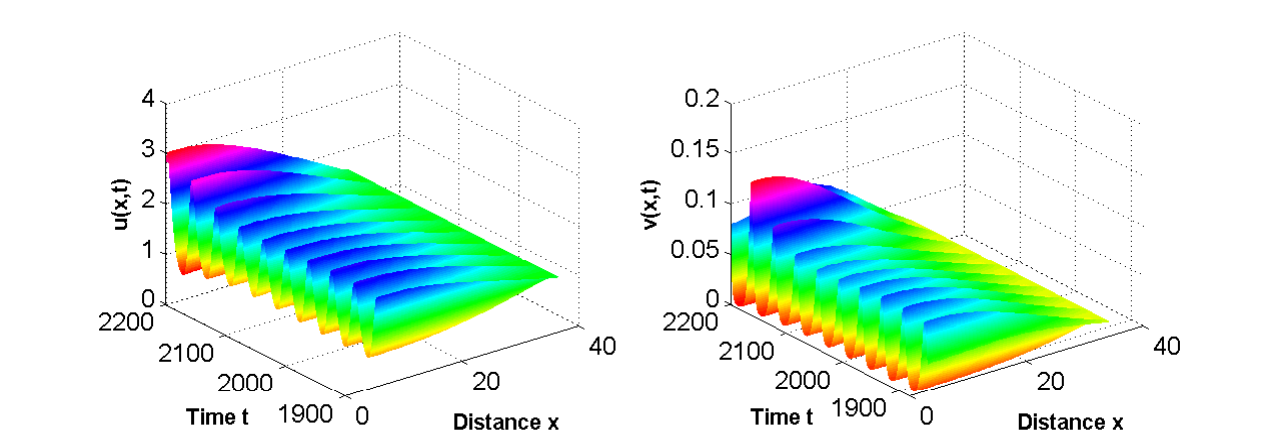}}\\
\centering{\includegraphics[width=0.8\textwidth]{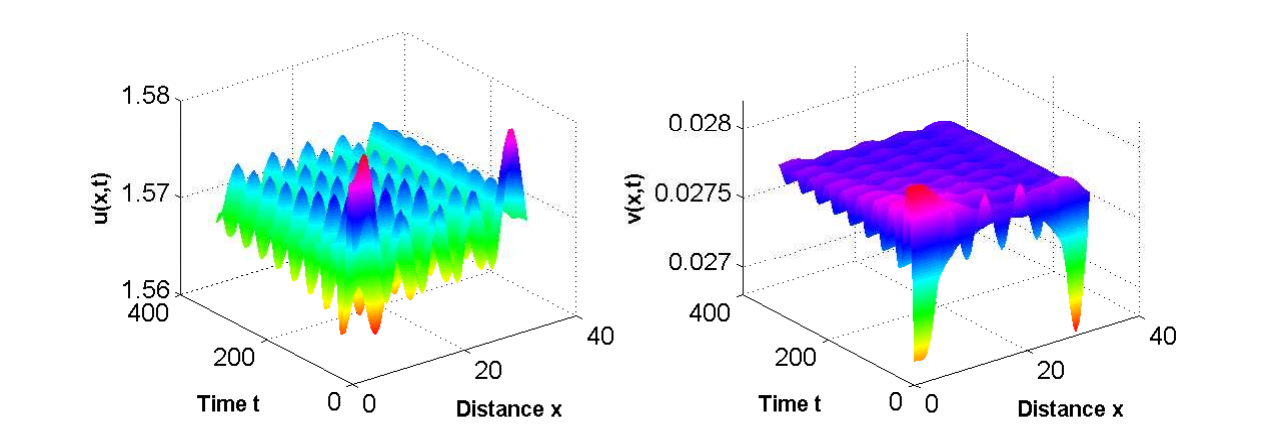}}
\caption{When $\tau=480$, the system (\ref{607a}) experiences two unstable transients.}
\label{e608}
\end{figure}

b) If we choose Holling-II functional response function $g(u)=\frac{p_0u}{1+c_1u}$,
then (\ref{523a}) has the following form:
\begin{eqnarray}\label{f622}
\begin{cases}
    u_t=d_{1}u_{xx}-q_{1}u_x +u\big[\frac{r_{0}}{1+Kv}-d-au(x,t-\tau)\big]-\frac{p_0u}{1+c_1u}v, \quad t>0, x\in(0,l\pi), \\
    v_t=d_{2}v_{xx}-q_{2}v_x +v\big(-r_{2}-mv\big)+\frac{cp_0uv}{1+c_1u},  \quad t>0, x\in(0,l\pi).
\end{cases}
\end{eqnarray}
Keeping all the parameters consistent with (\ref{e622}) and choosing $c_1=0.2$,  we can obtain the unique positive steady-state solution $u_{\ast}=2.2792$, $v_{\ast}=0.0098$.  Enlarge the range of
$r_0$ to [0.18,1], where (S1) and (S3) hold. Fig. \ref{d819} shows that $\tau_{n_0}^{j\pm}$ is monotonically decreasing with respect to $r_0$. When $r_0=0.2$ is fixed, $\tilde{\tau}=8.835$.
With the initial functions $u(x,0)=u_{\ast}+0.001\cos x$, $v(x,0)=v_{\ast}+0.001\cos x$, we find that the system (\ref{f622}) exhibits stable spatially nonhomogeneous periodic solutions (see Fig. \ref{h622}(a))
when $\tau=16>\tilde{\tau}$. For predators, the population mainly gathers at the upstream end, i.e. near $x=0$. The density of prey at the upstream and downstream ends is higher than that in the middle section, that is, they gather at $x=0$ and $x=l\pi$. In  Fig. \ref{b608}, both predators and prey gather at the upstream end. From this, it can be seen that even if the system has a stable periodic solution, the activity regions of prey and predator are significantly different due to differences in functional response functions and time delays. In addition, keeping the time delay constant, i.e. $\tau=16$, combined with the FF/FF boundary conditions, we find that predators eventually become extinct, and the number of prey varies periodically in Fig. \ref{h622}(b).
\begin{figure}[h]
\centering{\includegraphics[width=0.5\textwidth]{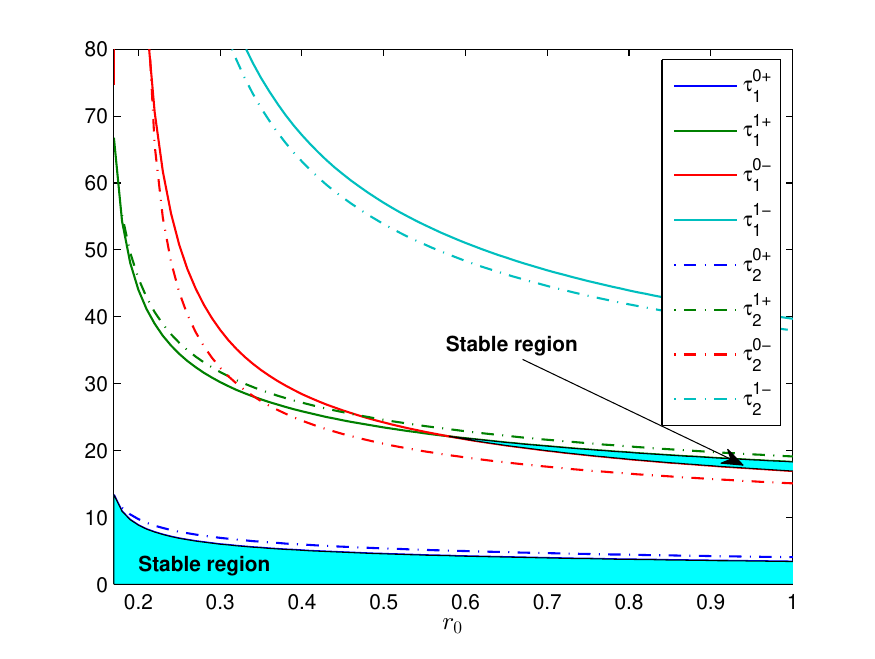}}
\caption{Partial Hopf bifurcation lines and stable regions on the $r_0-\tau$ plane of system (\ref{f622}).}
\label{d819}
\end{figure}

\begin{figure}[h]
\centering{(a)\includegraphics[width=0.7\textwidth]{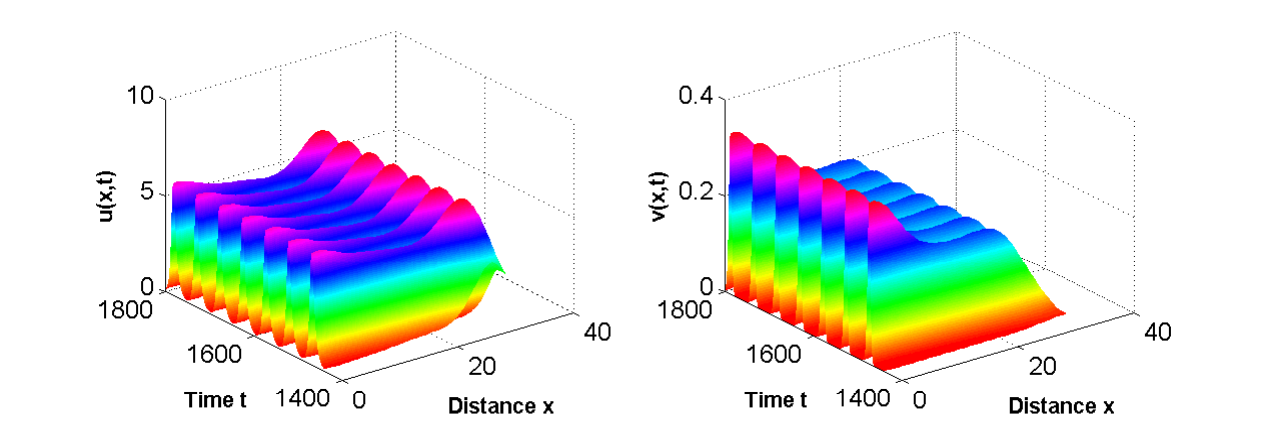}}
{(b)\includegraphics[width=0.7\textwidth]{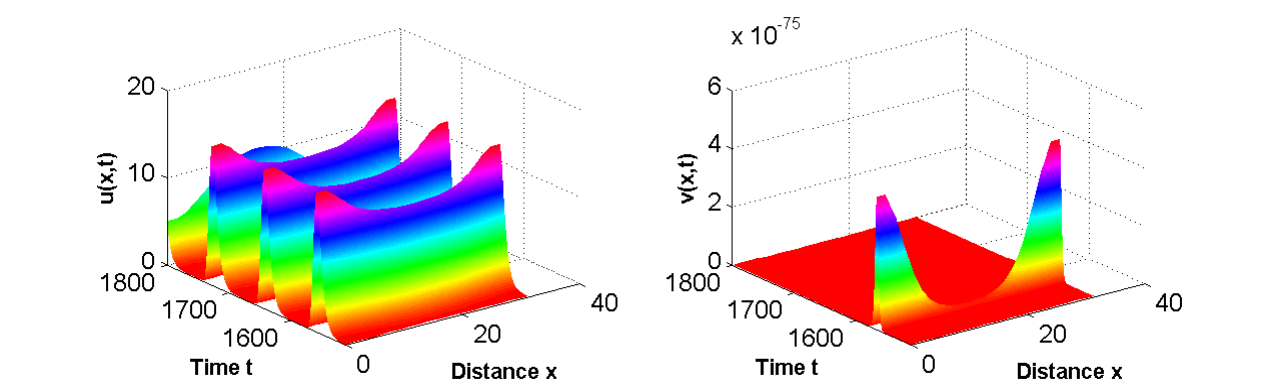}}
\caption{$\tau=16$. (a)CF/D boundary; (b)FF/FF boundary.}
\label{h622}
\end{figure}

\section{Conclusion}\label{b1112}
Firstly, we delve into the dynamics of a predator-prey model that incorporates fear effects in an advective environment. We have explored this model under two boundary conditions: CF/D and FF/FF. To obtain a deeper understanding, we calculate the eigenvalues and adjoint operators for each condition.
Next, our focus are on performing a linear stability analysis of the system. In our study, we select the maturation time delay $\tau$ as the bifurcation parameter and derive the normal form of Hopf bifurcation near the positive steady-state solution of the model.
Upon analyzing the numerical results, we have observed that the system exhibits various interesting phenomena for
different delays/boundary conditions. This includes the presence of stable spatial nonhomogeneous periodic solutions, unstable transients, and more. We have examined both linear and Holing-II functional response functions.
Moreover, under the FF/FF boundary condition, we have found that the system displays stable homogeneous periodic solutions, and even the system could ultimately lead to the extinction of the predator.

\section*{Acknowledgement}
This research is supported by Natural Science Foundation of Shandong Province (No. ZR2022QA075) and Natural Sciences and Engineering Research Council of Canada (No. RGPIN-2023-05976).


\begin{thebibliography}{0}%
\makeatletter
\providecommand \@ifxundefined [1]{%
 \@ifx{#1\undefined}
}%
\providecommand \@ifnum [1]{%
 \ifnum #1\expandafter \@firstoftwo
 \else \expandafter \@secondoftwo
 \fi
}%
\providecommand \@ifx [1]{%
 \ifx #1\expandafter \@firstoftwo
 \else \expandafter \@secondoftwo
 \fi
}%
\providecommand \natexlab [1]{#1}%
\providecommand \enquote  [1]{``#1''}%
\providecommand \bibnamefont  [1]{#1}%
\providecommand \bibfnamefont [1]{#1}%
\providecommand \citenamefont [1]{#1}%
\providecommand \href@noop [0]{\@secondoftwo}%
\providecommand \href [0]{\begingroup \@sanitize@url \@href}%
\providecommand \@href[1]{\@@startlink{#1}\@@href}%
\providecommand \@@href[1]{\endgroup#1\@@endlink}%
\providecommand \@sanitize@url [0]{\catcode `\\12\catcode `\$12\catcode
  `\&12\catcode `\#12\catcode `\^12\catcode `\_12\catcode `\%12\relax}%
\providecommand \@@startlink[1]{}%
\providecommand \@@endlink[0]{}%
\providecommand \url  [0]{\begingroup\@sanitize@url \@url }%
\providecommand \@url [1]{\endgroup\@href {#1}{\urlprefix }}%
\providecommand \urlprefix  [0]{URL }%
\providecommand \Eprint [0]{\href }%
\providecommand \doibase [0]{http://dx.doi.org/}%
\providecommand \selectlanguage [0]{\@gobble}%
\providecommand \bibinfo  [0]{\@secondoftwo}%
\providecommand \bibfield  [0]{\@secondoftwo}%
\providecommand \translation [1]{[#1]}%
\providecommand \BibitemOpen [0]{}%
\providecommand \bibitemStop [0]{}%
\providecommand \bibitemNoStop [0]{.\EOS\space}%
\providecommand \EOS [0]{\spacefactor3000\relax}%
\providecommand \BibitemShut  [1]{\csname bibitem#1\endcsname}%
\let\auto@bib@innerbib\@empty
\end{thebibliography}%


\begin{thebibliography}{99}
\bibitem{Lima1998Nonlethal}
Lima S L. Nonlethal effects in the ecology of predator-prey interactions. Bioscience 1998;48:25-34.
\bibitem{Creel2008Relationships}
Creel S, Christianson D. Relationships between direct predation and risk effects. Trends Ecol Evolut 2008;23:194-201.
\bibitem{Lima2009Predators}
Lima S L. Predators and the breeding bird: behavioral and reproductive flexibility under the risk of predation. Biol Rev 2009;84:485-513.
\bibitem{Wang2016Modelling}
Wang X, Zanette L, Zou X. Modelling the fear effect in predator-prey interactions. J Math Biol 2016; 73:1179-1204.

\bibitem{Duan2019Hopf}
D. Duan, B. Niu and J. Wei, Hopf-Hopf bifurcation and chaotic attractors in a delayed diffusive
predator-prey model with fear effect, Chaos, Solitons, Fractals, 123 (2019), 206-216.
\bibitem{Wang2022Spatiotemporal}
Wang C, Yuan S, Wang H. Spatiotemporal patterns of a diffusive prey-predator model with spatial memory and pregnancy period in an intimidatory environment. J Math Biol 2022;84(3):1-36.

\bibitem{Panday2018Stability}
P. Panday, N. Pal, S. Samanta, J. Chattopadhyay, Stability and bifurcation analysis of a three-species food chain model with fear, Int. J Bifurcat. Chaos. 28 (2018) 1850009.

\bibitem{Han2020Cross}
R. Han, L.N. Guin, B. Dai, Cross-diffusion-driven pattern formation and
selection in a modified Leslie-Gower predator-prey model with fear effect, J. Biol. Syst. 28 (2020) 27-64.
\bibitem{Wu2018Dynamics}
S. Wu, J. Wang, J. Shi, Dynamics and pattern formation of a diffusive predator-prey model with
predator-taxis, Math. Mod. Meth. Appl. S. 28 (2018) 2275-2312.
\bibitem{Liu2021Influence}
Q. Liu, D. Jiang, Influence of the fear factor on the dynamics of a stochastic predator-prey model, Appl. Math. Lett. 112 (2021) 106756.
\bibitem{Zhang2023Dynamics}
X. Zhang, H. Zhu, Q. An, Dynamics analysis of a diffusive predator-prey model with spatial memory and nonlocal fear effect, J. Math. Anal. Appl. 525 (2023) 127123.


\bibitem{Jin2014Seasonal}
Jin Y, Hilker FM, Steffler PM, Lewis MA. Seasonal invasion dynamics in a spatially heterogeneous river with fluctuating flows. Bull Math Biol 2014;76(7):1522-1565.
\bibitem{Lutscher2006Effects}
Lutscher F, Lewis MA, McCauley E. Effects of heterogeneity on spread and persistence in rivers. Bull Math Biol 2006;68(8):2129-2160.
\bibitem{Lutscher2007Spatial}
Lutscher F, McCauley E, Lewis MA. Spatial patterns and coexistence mechanisms in systems with unidirectional flow. Theor Popul Biol 2007;71(3):267-277.
\bibitem{Speirs2001Population}
Speirs DC, Gurney WSC. Population persistence in rivers and estuaries. Ecology 2001;82(5):1219-1237.
\bibitem{Chen2018Hopf}
Chen S, Lou Y, Wei J. Hopf bifurcation in a delayed reaction-diffusion-advection population model. J. Differ. Equ., 2018, 264(8):5333-5359.
\bibitem{Chen2020Bifurcation}
Chen S, Wei J, Zhang X. Bifurcation analysis for a delayed diffusive
logistic population model in the advective heterogeneous environment. J. Dyn. Differ. Equ., 2020, 32(2): 823-847.
\bibitem{Bin2021Bifurcation}
Bin H, Zhang H, Wei J. Bifurcation analysis in a delayed reaction-diffusion-advection food-limited system. Appl. Math. Lett., 2021, 120:107332.
\bibitem{Lou2014Evolution}
Lou Y, Lutscher F. Evolution of dispersal in open advective environments.
J. Math. Biol. 2014, 69(6): 1319-1342.
\bibitem{Jin2021Hopf}
Z. Jin, R. Yuan, Hopf bifurcation in a reaction-diffusion-advection equation
with nonlocal delay effect, J. Differ. Equ., 271 (2021) 533-562.
\bibitem{Li2022Hopf}
Z. Li, B. Dai, R. Han, Hopf bifurcation in a reaction-diffusion-advection
population model with distributed delay, Int. J. Bifurcat. Chaos., 32 (2022)
2250247.
\bibitem{Liu2023Steady}
D. Liu, W. Jiang, Steady-state bifurcation and Hopf bifurcation in a
reaction-diffusion-advection system with delay effect, J. Dyn. Differ. Equ.,
https://doi.org/10.1007/s10884-022-10231-5.
\bibitem{Kubrusly2012Spectral}
C.S. Kubrusly, Spectral Theory of Operators on Hilbert Spaces, Springer, New York, 2012.
\bibitem{Hale1977Theory}
J. Hale, Theory of Functional Differential Equations, Springer-Verlag, Berlin, 1977.
\bibitem{Wu1996Theory}
J. Wu, Theory and Applications of Partial Functional Differential Equations, Springer-Verlag, New York, 1996.
\bibitem{Faria2002Trans}
T. Faria, Normal forms and Hopf bifurcation for partial differential equations with
delays, Trans. Amer. Math. Soc., 2000,352: 2217-2238.
\bibitem{Bin2023}
H. Bin, D. Duan, J. Wei. Bifurcation analysis of a reaction-diffusion-advection predator-prey system with delay. Math Biosci. Eng., 2023, 20(7): 12194-12210.
\end{thebibliography}

\end{document}